\documentclass{amsart}
\usepackage{amsmath}
\usepackage{amssymb}
\usepackage[all]{xy}
\usepackage{enumerate}

\theoremstyle{plain}
\newtheorem{thm}{Theorem}[section]
\newtheorem{lemma}[thm]{Lemma}
\newtheorem{cor}[thm]{Corollary}
\newtheorem{prop}[thm]{Proposition}

\theoremstyle{definition}

\newtheorem{remark}[thm]{Remark}

\begin{document}

\title{Springer Isomorphisms In Characteristic $p$}

\author[Paul Sobaje]
{Paul Sobaje}

\begin{abstract}

\sloppy{
Let $G$ be a simple algebraic group over an algebraically closed field of characteristic $p$, and assume that $p$ is a separably good prime for $G$.  Let $P$ be a parabolic subgroup whose unipotent radical $U_P$ has nilpotence class less than $p$.  We show that there exists a particularly nice Springer isomorphism for $G$ which restricts to a certain canonical isomorphism $\textup{Lie}(U_P) \xrightarrow{\sim} U_P$ defined by J.-P. Serre.  This answers a question raised both by G. McNinch in \cite{M2}, and by J. Carlson \textit{et. al} in \cite{CLN}.  For the groups $SL_n, SO_n$, and $Sp_{2n}$, viewed in the usual way as subgroups of $GL_n$ or $GL_{2n}$, such a Springer isomorphism can be given explicitly by the Artin-Hasse exponential series.
}

\end{abstract}

\maketitle

Let $k$ be an algebraically closed field of characteristic $p>0$.  Let $G$ be a simple algebraic group over $k$, by which we mean that $G$ is non-commutative and that the trivial group is the only connected normal reduced algebraic subgroup (this is sometimes referred to as quasi-simple or almost-simple).  Assume that $p$ is \textit{good} for $G$.  Let $P$ be a parabolic subgroup whose unipotent radical $U_P$ has nilpotence class less than $p$, and let $\mathfrak{g}, \mathfrak{u}_P$ be the Lie algebras of $G$ and $U_P$ respectively.  Denote by $\mathcal{N}(\mathfrak{g})$ the nilpotent variety of $\mathfrak{g}$, and by $\mathcal{U}(G)$ the unipotent variety of $G$.

An argument due to J.-P. Serre, given in \cite{Ser} and more elaborately in \cite{Sei}, demonstrates that there is a canonical $P$-equivariant isomorphism $\varepsilon_P: \mathfrak{u}_P \xrightarrow{\sim} U_P$ which is uniquely determined by a few desirable properties which we will detail below.  If the prime $p$ is also \textit{separably good} for $G$, which means that is it good and that it does not divide the order of the fundamental group of $G$, then in this case T.A. Springer proved that there exists a $G$-equivariant isomorphism between $\mathcal{N}(\mathfrak{g})$ and $\mathcal{U}(G)$ \cite{Sp2}; such a map is known as a `Springer isomorphism.'  There are in general many Springer isomorphisms for a given $G$; Serre has shown in the appendix to \cite{M2} that they can be parameterized by a variety of dimension equal to the rank of $G$.  However, every Springer isomorphism for $G$ will restrict to a $P$-equivariant isomorphism between $\mathfrak{u}_P$ and $U_P$ \cite[Remark 10]{M2}, and G. McNinch observed in Remark 27 of \textit{loc. cit.} that if $p \ge h$, the Coxeter number of $G$, then there is always some Springer isomorphism whose restriction yields $\varepsilon_P$ for any parabolic subgroup $P$, each of which has unipotent radical with nilpotence class less than $p$.  The author then asks if this remains true when $p<h$.  This same question also appears in work by J. Carlson, Z. Lin, and D. Nakano \cite[\S 2.7]{CLN}, where the question is posed because an answer in the affirmative would give an immediate proof of, and in fact extend, \cite[Theorem 3]{CLN} (see Remark \ref{CLNexplained} below).

In this paper we show that when $p$ is separably good for $G$ there is indeed a Springer isomorphism which restricts to $\varepsilon_P$ on every parabolic subgroup $P$ whose unipotent radical has nilpotence class less than $p$.  Thanks to a result due to McNinch in \cite{M}, we show that if $G$ is a classical simple subgroup of $GL_n$ then a particular Springer isomorphism with this property can be given explicitly by the Artin-Hasse exponential series.  On the other hand, proving that these isomorphisms exist in general relies on work by G. Seitz in \cite{Sei} and \cite{Sei2} on abelian unipotent overgroups of unipotent elements in $G$.  These papers in turn depend on and sharpen earlier results by D. Testerman \cite{T} and R. Proud \cite{P}.

To further highlight the relevance of these results, we point out that if $H$ is a simple algebraic group over an algebraically closed field of characteristic $0$, then there is always a ``preferred" Springer isomorphism which is given by the exponential map.  That is, $H$ may be embedded in some $GL_n$, and the exponential map on nilpotent matrices in $\mathfrak{gl}_n$ restricts to a Springer isomorphism for $H$ (as follows from \cite[Proposition 7.1]{M}).   Theorem \ref{main} can therefore be seen as an attempt to find a suitable analogue of the exponential map in prime characteristic (although there is no claim that when $p<h$ the properties in the theorem uniquely specify an isomorphism).  For instance, if $0 \ne X \in \mathcal{N}(\mathfrak{h})$ then the exponential map takes the line $kX$ to a one-parameter additive subgroup of $H$, and Theorem \ref{main}(3) gives a characteristic $p$ generalization of this, replacing one-parameter subgroups with Witt groups.  The exponential map also defines a group isomorphism between the Lie algebra of a Borel subgroup $B$ (as a group under the Baker-Campbell-Hausdorff formula) and $B$, while Theorem \ref{main}(1) effectively gives the strongest version of this in positive characteristic.

\section{Preliminaries}

\subsection{Notation and Conventions}

Throughout $G$ will denote a simple algebraic group over an algebraically closed field $k$ of characteristic $p>0$.  Fix a maximal torus of $G$, and let $\Phi$ denote the root system of $G$ with respect to this choice.  We say that $p$ is a \textit{good} prime for $G$ if $p$ does not divide the coefficients of the highest root of $\Phi$ with respect to some choice of simple roots.  This means that $p>2$ if $G$ is of type $B,C$ or $D$; $p>3$ if $G$ is of type $E_6,E_7,F_4$ or $G_2$; and $p>5$ if $G=E_8$ .  We follow J. Pevtsova and J. Stark in saying that $p$ is \textit{separably good} if it does not divide the order of the fundamental group of $G$.  This is already satisfied by good primes in all types except for type $A$.  We note that $p$ is separably good if and only if $p$ is good for $G$ and the covering $G_{sc} \rightarrow G$ is a separable morphism, where $G_{sc}$ denotes the simply-connected group isogenous to $G$ (this matches the formulation in \cite[Definition 2.2]{PS}).

We denote  by $\mathcal{U}(G)$ the unipotent variety of $G$, and by $\mathcal{N}(\mathfrak{g})$ the nilpotent variety of its Lie algebra.  The conjugation action of $G$ on itself induces an action on both $\mathcal{U}(G)$ and $\mathcal{N}(\mathfrak{g})$.  Each variety is irreducible, and each has a unique open orbit under the action of $G$, referred to in both contexts as the \textit{regular} orbit.  An element in the regular nilpotent (resp. unipotent) orbit is called a regular nilpotent (resp. unipotent) element.  The subvariety of $p$-unipotent elements in $G$ will be denoted by $\mathcal{U}_1(G)$, while $\mathcal{N}_1(\mathfrak{g})$ denotes the $[p]$-nilpotent variety of $\mathfrak{g}$, where $x \mapsto x^{[p]}$ is the restriction map on $\mathfrak{g}$.  We also will refer to $\mathcal{N}_1(\mathfrak{g})$ as the restricted nullcone.  An element $X$ will be said to have nilpotent order $p^m$ if $X^{[p^m]}=0$ and $X^{[p^{m-1}]} \ne 0$.

Let $P$ be a parabolic subgroup of $G$ whose unipotent radical $U_P$ has nilpotence class less than $p$.  We will follow \cite{CLNP} in referring to $P$ as a \textit{restricted} parabolic subgroup of $G$ (we note that this formulation is stated differently than in \textit{loc. cit.}, though it is equivalent in the cases we are considering).

For any affine algebraic group $H$ over $k$, we denote by $k[H]$ its coordinate algebra.  This is a commutative Hopf algebra over $k$.  The subgroup $H^0$ denotes the identity component of $H$.

We recall that an abstract group $\Gamma$ is called nilpotent if its descending central series has finite length, in which case this length is known as the nilpotence class of $\Gamma$.

\subsection{Springer Isomorphisms}

For $p$ separably good for $G$, Springer \cite{Sp2} first proved that there exists a $G$-equivariant homeomorphism between $\mathcal{N}(\mathfrak{g})$ and $\mathcal{U}(G)$.  This was later shown to be an isomorphism due to the normality of both varieties.  Springer's method will not be recounted here (but is presented clearly in \cite[\S 6.21]{H}).  Rather, we aim to justify the claim that finding a Springer isomorphism for $G$ (in separably good characteristic) reduces to finding a regular nilpotent element $X$ and a regular unipotent element $u$ whose centralizers $C_G(X)$ and $C_G(u)$ are equal.

First, we noted earlier that the regular unipotent and the regular nilpotent orbits are open, and in fact both have the property that their complements are of codimension at least $2$ in their respective varieties.  The normality of $\mathcal{N}(\mathfrak{g})$ and $\mathcal{U}(G)$ then allows for any isomorphism between these orbits to be extended uniquely to an isomorphism between $\mathcal{N}(\mathfrak{g})$ and $\mathcal{U}(G)$.  Thus, finding a Springer isomorphism reduces to finding a $G$-equivariant isomorphism between the regular orbits.  This, however, can be reduced to finding $X$ and $u$ as above.  The key result used in this last step is that the $G$-orbit of $X$ is isomorphic to the quotient $G/C_G(X)$ (see \S 2.2 and \S 2.9 of \cite{J}, and note that in the case of type $A$, we obtain the result for $SL_n$ by instead working over $GL_n$, as the unipotent and nilpotent varieties of $SL_n$ are the same as those of $GL_n$).

We now gather some important results about Springer isomorphisms which are found in Remark 27 and the Appendix of \cite{M2}, and \cite[Theorem E]{MT}.

\begin{thm}\label{differential}  \cite{M2} \cite{MT} Let $\phi$ be a Springer isomorphism from $\mathcal{N}(\mathfrak{g})$ to $\mathcal{U}(G)$. 
\begin{enumerate}
\item For any parabolic subgroup $P \le G$ with unipotent radical $U_P$, $\phi$ restricts to an isomorphism between $\mathfrak{u}_P$ and $U_P$.
\item The restriction of $\phi$ to $\mathfrak{u}_P$ has differential sending $\mathfrak{u}_P$ to $\mathfrak{u}_P$, and this map is a scalar multiple of the identity.
\item If $\phi^{\prime}$ is any other Springer isomorphism for $G$, then $\phi$ and $\phi^{\prime}$ give the same bijection between nilpotent and unipotent orbits.
\end{enumerate}
\end{thm}

\begin{remark}\label{sameorder}
McNinch showed in \cite[Theorem 35]{M3} that there is always some Springer isomorphism $\rho$ which satisfies $\rho(X^{[p]})=\rho(X)^p$, hence for any Springer isomorphism $\phi$ Theorem \ref{differential}(3) implies that $X$ has nilpotent order $p^m$ if and only if $\phi(X)$ has unipotent order $p^m$.
\end{remark}

\begin{remark}\label{tangentmap}
If $\phi$ is a Springer isomorphism and $B$ a Borel subgroup of $G$, then in particular $\phi$ restricts to a $B$-equivariant isomorphism between the smooth varieties $\mathfrak{u}_B$ and $U_B$, and the tangent map at $0$ of this restriction is given by multiplication by some $c \in k^{\times}$ (see \cite[\S 5.5]{MT} for more).  Note that $c$ is independent of the choice of $B$, and by abuse of terminology we shall refer to this scalar map as ``the tangent map of $\phi$."
\end{remark}

\subsection{A Canonical Exponential Map For Restricted Parabolics}

Let $p$ be good for $G$, and suppose that $P$ is a restricted parabolic subgroup of $G$.  In \cite[Proposition 5.3]{Sei} (credited by the author to Serre), a $P$-equivariant isomorphism $\varepsilon_P: \mathfrak{u}_P \xrightarrow{\sim} U_P$ is obtained by base-changing the usual exponential isomorphism in characteristic $0$.  More specifically, we may assume that $P$ is a standard parabolic subgroup of $G$.  Then $G,P,$ and $U_P$ are defined over $\mathbb{Z}$, and one can show that the exponential isomorphism $\mathfrak{u}_{P,\mathbb{Q}} \xrightarrow{\sim} U_{P,\mathbb{Q}}$ is defined over $\mathbb{Z}_{(p)}$, and hence can be base-changed to $k$.  This isomorphism can be identified over $k$ according to the following properties (see \cite[Ch. 2, \S 6]{B} for Baker-Campbell-Hausdorff formula):

\begin{enumerate}
\item It is $P$-equivariant.
\item There is a group structure on $\mathfrak{u}_P$ given by the Baker-Campbell-Hausdorff formula, and $\varepsilon_P$ is an isomorphism of algebraic groups with respect to this structure on $\mathfrak{u}_P$.
\item The tangent map is the identity.
\end{enumerate}

\bigskip
We note that Theorem \ref{differential}(2) indicates that this last condition, on its own, is not all that unique.  However, when coupled with the first property it specifies uniquely the isomorphism between $\textup{Lie}(U_{\alpha})$ and $U_{\alpha}$, where $U_{\alpha}$ is a root subgroup of $U_P$.  Since the root subgroups generate $U_P$, there is at most one map satisfying properties (1)-(3).  This essentially recapitulates an argument given in the proof of \cite[Proposition 5.2]{Sei}.

\subsection{The Artin-Hasse Exponential}

For a given prime $p$, the Artin-Hasse exponential is the power series $E_p(t)$ defined by

$$E_p(t) = \textup{exp}\left(t + \frac{t^p}{p} + \frac{t^{p^2}}{p^2} + \cdots \right)$$

\noindent This power series evidently lies in $\mathbb{Q}[\![ t ]\!]$, however one can actually prove that $E_p(t) \in \mathbb{Z}_{(p)}[\![ t ]\!]$ (see \cite[Proposition 1]{D} for a more general fact).  Let $C_i$ denote the coefficient of $t^i$ in $E_p(t)$, and $c_i$ its image in $\mathbb{F}_p$ under the unique homomorphism from $\mathbb{Z}_{(p)}$ to $\mathbb{F}_p$.  We obtain in this way an element $e_p(t) \in \mathbb{F}_p [\![ t ]\!] \subseteq k[\![ t ]\!]$, where the coefficient of $t^i$ in $e_p(t)$ is $c_i$.

We note that as elements in $\mathbb{Q}[\![ t ]\!]$, the series $E_p(t)$ will agree with the series $\textup{exp}(t)$ over its first $p$ coefficients.  Thus $C_i = 1/i!$ for $i<p$.  We also must point out that some sources, for example \cite{Ser2}, define the Artin-Hasse exponential to be the series 

$$F_p(t) = \textup{exp}\left(-\left(t + \frac{t^p}{p} + \frac{t^{p^2}}{p^2} + \cdots \right)\right)$$

In particular, this definition is the one employed by McNinch in  \cite[Proposition 7.5]{M}, a result which we will later use.  As observed in \cite{D}, this series is just the inverse of $E_p(t)$, in the sense that $F_p(t)E_p(t) = 1 \in \mathbb{Z}_{(p)}[\![ t ]\!]$.

\subsection{Witt Groups and Connected Abelian Unipotent Groups}\label{witt}

We will need some basic information about Witt groups, and more generally about connected abelian unipotent groups.  A standard source is \cite{Ser2}, we have also benefited from the exposition in \cite{P} and \cite[\S 3]{M}.

Let $\mathcal{W}_m$ denote the group of Witt vectors of length $m$ over $k$.  This is a connected abelian unipotent group which is isomorphic as a variety to $\mathbb{A}^m$.  We can therefore put coordinates on $\mathcal{W}_m$ so that an element can be written as $(a_0,a_1,\ldots,a_{m-1})$, and accordingly $k[\mathcal{W}_m] \cong k[t_0,t_1,\ldots,t_{m-1}]$.

The following theorem uses the Artin-Hasse exponential series to explicitly describe $\mathcal{W}_m$ as a matrix group.

\begin{thm}\label{wittmorphism}\cite[Theorem 7.4]{P} \cite[\S V.16]{Ser2}
Let $X \in \mathfrak{gl}_n$ be such that $X^{p^m}=0$ and $X^{p^{m-1}}\ne0$.  Then the map $f: \mathcal{W}_m \rightarrow GL_n$ given by $$(a_0,a_1,\ldots,a_{m-1}) \mapsto e_p(a_0X)e_p(a_1X^p)\cdots e_p(a_{m-1}X^{p^{m-1}})$$ is an isomorphism of algebraic groups onto its image.
\end{thm}

This realization of $\mathcal{W}_m$ makes a few of its properties clear.  First, writing the group operation of $\mathcal{W}_m$ in multiplicative notation, we have $$(a_0,a_1,\ldots,a_{m-1})^p = (0,a_0^p,a_1^p,\ldots,a_{m-2}^p).$$  Second, let $\frac{d}{dt_i} \in \textup{Lie}(\mathcal{W}_m)$ be dual to $t_i$.  Then we observe, as is also done in \cite[Lemma 3.3(2)]{M}, that $$df\left(\frac{d}{dt_i}\right)=X^{p^i}$$ thus $$\left(\frac{d}{dt_i}\right)^{[p]}=X^{p^{i+1}}= \frac{d}{dt_{i+1}}.$$  Finally, we observe that the elements of order $p^j$ are those of the form $$(0,\ldots,0,a_{m-j},\ldots,a_{m-1}), \text{ where } a_{m-j} \ne 0.$$

Let $H$ now be an arbitrary connected abelian unipotent group over $k$.  For each $j \ge 1$, let $H^{p^j}$ denote the subgroup generated by all $p^j$-th powers of elements in $H$, and $H_{p^j}$ the subgroup of all elements in $H$ having order dividing $p^j$.  By \cite[VII.10]{Ser2} we know that $H$ is isogenous to a direct product of Witt groups, and in the special case that $H$ has dimension $m$ and $m$ is also the smallest integer for which $H^{p^m} = 0$ then $H$ is isogenous to $\mathcal{W}_m$.  

\subsection{Unipotent Overgroups And Centralizers}

Again let $G$ be simple and $p$ good for $G$.  Let $u \in G$ be unipotent of order $p^r$.  When $r=1$ it was shown by Testerman \cite{T} that $u$ lies in a closed simple subgroup of $G$ of type $A_1$ (isomorphic to either $PSL_2$ or $SL_2$).  In particular this shows that $u$ is always contained in a one-parameter additive subgroup of $G$.  Seitz then extended this in \cite{Sei}, showing that that there is a canonical one-parameter additive subgroup of $G$ containing $u$.  This one-parameter subgroup is referred to as the saturation of $u$.  

Saturation is achieved (or specified) as follows.  Let $A$ be a simple subgroup of $G$ of type $A_1$, and let $T_A$ be a maximal torus of $A$.  Then $A$ is said to be a \textit{good} $A_1$ subgroup if $\mathfrak{g}$, as a $T_A$-module, only has weights which are $\le 2p-2$.  We then have the following:

\begin{thm}\label{mono}\cite{Sei}
Let $u \in G$ be unipotent of order $p$.  Then there is a unique monomorphism $\varphi_u: \mathbb{G}_a \rightarrow G$ with image contained in a good $A_1$ and satisfying $\varphi_u(1) = u$.
\end{thm}

When $r>1$, it is clear that $u$ cannot be contained in a subgroup isomorphic to $\mathbb{G}_a$.  However, it was shown by Proud \cite{P} that $u$ can be embedded in a subgroup of $G$ isomorphic to $\mathcal{W}_r$.  This result was again refined by Seitz in \cite{Sei2}, and relies on the following result about centralizers of unipotent elements which was first established by Proud in an unpublished manuscript.

\begin{thm}\label{center} \cite{Sei2}
Let $u$ be a unipotent element of $G$.  Then $$Z(C_G(u))=Z(G) \times Z(C_G(u))^0,$$ and $Z(C_G(u))^0$ is the unipotent radical of $Z(C_G(u))$.
\end{thm}

Following Seitz, we call a one-dimensional torus $T$ of $G$ $u$-distinguished if there is a nilpotent element $X \in \mathfrak{g}$ such that $X$ is a weight vector for $T$ of weight $2$, $C_G(X) = C_G(u)$, and $T$ is contained in the derived subgroup of a Levi subgroup of $G$ for which $u$ is distinguished.  Such a torus is also the image of an \textit{associated cocharacter} of the nilpotent element $X$ (see \cite[\S 5.3]{J} for an explanation of this terminology).  Seitz then proved the following:

\begin{thm}\label{overgroup}\cite{Sei2}
For a fixed $u$-distinguished torus $T$, there is a unique subgroup $W \le Z(C_G(u))^0$ containing $u$ which is isogenous to $\mathcal{W}_r$ and such that $T$ acts on $W$ without fixed points.  The action of $T$ on $W/W^p$ is by weight $2$, and $W^{p^{r-1}}$ is the saturation of $u^{p^{r-1}}$.
\end{thm}

\noindent We highlight a few further details about this result:

\noindent \begin{enumerate}

\item According to \cite[Lemma 2.7]{Sei2} we may assume that the $X$ above is an element of $\textup{Lie}(W)$.

\item When $u$ has order $p$ then $W$ is the saturation of $u$ and is therefore canonical.  In general the overgroup $W$ depends on the choice of $T$ \cite[\S 4.3]{Sei2}.

\item Any group $W$ which is isogenous to a Witt group and on which a one-dimenional torus $T$ acts without fixed points is referred to by Seitz as being \textit{$T$-homocyclic} group.  If $W$ is $T$-homocyclic and isogenous to $\mathcal{W}_r$, then for each $1 \le j \le r$ we have $W^{p^j} = W_{p^{r-j}}$ (see remarks just above \cite[Theorem 1]{Sei2}).  There are groups isogenous to $\mathcal{W}_r$ which do not have this property \cite[VII.11]{Ser2}.

\end{enumerate}

\bigskip
Finally, we prove a useful lemma for groups of exceptional type which will be needed in proving the ``main theorem" found in Section 4.

\begin{lemma}\label{isoismorphic}
Suppose that $G$ is of exceptional type and that $u$ is a unipotent element of order $p^2$.  Fix a $u$-distinguished torus $T$ and let $W$ be as in Theorem \ref{overgroup}.  Then $W$ is isomorphic to $\mathcal{W}_2$.
\end{lemma}

\begin{proof}
In \cite[VII.11]{Ser2} it is shown that there are two invariants which determine all connected abelian unipotent groups of dimension $2$ up to isomorphism.  The first invariant of $W$ is the isomorphism class of the finite subgroup $W^p/W_p$, which from the comments above must be the trivial group.  The second invariant comes from the bijective algebraic group homomorphism $W/W^p \rightarrow W^p$ given by sending $w$ to $w^p$.  Putting coordinates on $W$, this $p$-th power map takes the form $(a,b)^p = (0,a^{p^h})$, where $h \ge 1$.  The integer $h$ is then the second invariant of $W$.

Seitz proves in \cite[Lemma 4.3]{Sei2} that $T$ acts with weight 2 on $W/W^p$, and with weight $2p$ on $W^p$.  We claim that this implies that $h=1$ for $W$.  Indeed, if $t \in T$ and $w \in W$, then it is clear that $t.w^p = (t.w)^p$.  Now put coordinates on $W$ so that $w = (a,b)$, and fix an isomorphism from $k^{\times}$ to $T$ so that if $t$ is the image of $c \in k^{\times}$ then $t.w = (c^2a,b^{\prime})$.  We then have that $$t.(0,a^{p^h}) = t.w^p = (t.w)^p = (0,(c^2a)^{p^h}) = (0,c^{2p^h}a^{p^h}).$$  On the other hand, if $T$ acts by weight $2p$ on $W^p$, then we see that $t.(0,a^{p^h}) = (0,c^{2p}a^{p^h})$.  Thus we must have that $2p^h = 2p$, so that $h=1$.

From the explicit description of Witt groups given in the previous section it is clear that $\mathcal{W}_2$ has these same invariants, therefore that $W \cong \mathcal{W}_2$.
\end{proof}

\section{Existence In General Type}

Let $G$ be simple and $p$ separably good for $G$.  Suppose that $\phi$ is a Springer isomorphism which restricts to $\varepsilon_P$ on all unipotent radicals of restricted parabolic subgroups.  Then it follows that for every $X \in \mathcal{N}_1(\mathfrak{g})$ this isomorphism $\phi$ maps the line $kX \subseteq \mathcal{N}_1(\mathfrak{g})$ to a one-parameter additive subgroup of $G$.  As it turns out, up to scalar multiplication this property is also a sufficient condition for $\phi$ to restrict to $\varepsilon_P$ (see Remark \ref{tangentmap} for an explanation of the terminology ``the tangent map of $\phi$").

\begin{prop}\label{sufficient}
If there exists a Springer isomorphism $\phi$ with tangent map the identity and having the property that for every $X \in \mathcal{N}_1(\mathfrak{g})$ the one-dimensional closed subset $\phi(kX)$ is a one-parameter additive subgroup of $G$, then for every restricted parabolic subgroup  $P \le G$ the map $\phi$ restricts to $\varepsilon_P$ on $\mathfrak{u}_P$.
\end{prop}

\begin{proof}
Let $P$ be restricted, and $X \in \mathfrak{u}_P$.  We have that $\phi$ restricts to a variety isomorphism between $kX$ and its image in $G$, the latter a one-parameter subgroup by assumption, so there exists a group isomorphism $\varphi$ from $\mathbb{G}_a$ to $\phi(kX)$ for which the map $\gamma:\mathbb{G}_a \rightarrow \mathbb{G}_a$ given by $\gamma(s) = \varphi^{-1}(\phi(sX))$ defines a variety automorphism of $\mathbb{G}_a$.  Since $\phi(0\cdot X)=1 = \varphi(0)$, it follows that $\gamma(0)=0$.  But a variety automorphism of $\mathbb{G}_a$ is of the form $s \mapsto b\cdot s+c$ for some $b,c \in k$ where $b \ne 0$, thus if $\gamma(0)=0$ it must in fact be a group automorphism.  Therefore, we see that the map sending $s$ to $\phi(sX)$ for all $s \in \mathbb{G}_a$ defines a monomorphism from $\mathbb{G}_a$ to $G$.

Let $T$ be a one-dimensional torus of $G$ which is the image of an associated cocharacter of $X$.  Then $u=\phi(X)$ is a $p$-unipotent element in $G$, and as $C_G(X)=C_G(u)$, we have that $T$ is $u$-distinguished.  For every $0\ne s\in k$ we have that $C_G(sX)=C_G(X)=C_G(u)$, hence $\phi(kX)$ is a one-parameter subgroup of $Z(C_G(u))$, and by Theorem \ref{center} it follows that $\phi(kX) \subseteq Z(C_G(u))^0$.  But this one-parameter subgroup is $T$-stable since $T$ stabilizes $kX$ and $\phi$ is $T$-equivariant.  By Theorem \ref{overgroup}, $\phi(kX)$ is therefore the saturation of $u$, and since $u = \phi(X)$, it follows that the unique monomorphism $\varphi_u$ in Theorem \ref{mono} is given by $\varphi_u(s) = \phi(sX)$.  Since $\phi$ is a Springer isomorphism and $X \in \mathfrak{u}_P$, then by Theorem \ref{differential}(1) the saturation of $u$ is a subgroup of $U_P$.  The argument in the proof of \cite[Proposition 5.5]{Sei} then applies and shows that there is some $Y \in \mathfrak{u}_P$ such that $\varphi_u(s) = \varepsilon_P(sY)$.  As both $\phi$ and $\varepsilon_P$ have tangent map the identity, the equality $\varphi_u(s) = \varepsilon_P(sY)$ implies on the one hand that $d\varphi_u(\frac{d}{dt})=Y$, while we see that $d\varphi_u(\frac{d}{dt})=X$ from the fact that $\varphi_u(s)=\phi(sX)$.  Therefore $Y=X$ and $\varepsilon_P(X)=\varphi_u(1)=\phi(X)$.  Since $P$ and $X$ were arbitrary, this finishes the proof.
\end{proof}

We now use Theorem \ref{overgroup} to construct a Springer isomorphism which will satisfy the hypotheses in the previous proposition.  We remind the reader that as pointed out in the remarks following Theorem \ref{overgroup}, if $T$ is a $u$-distinguished torus and $W$ is the unique $T$-homocyclic subgroup of $C_G(U)^0$ containing $u$, then $\textup{Lie}(W)$ contains a $T$-weight vector of weight $2$ having the same centralizer in $G$ as does $u$.

\begin{prop}\label{additive}
Let $u$ be a regular unipotent element in $G$, let $T$ be a $u$-distinguished torus, and let $W \subseteq C_G(U)^0$ be the unique $T$-stable subgroup containing $u$.  Let $X \in \textup{Lie}(W)$ be a $T$-weight vector of weight $2$ such that $C_G(X)=C_G(u)$, and let $\phi$ be the Springer isomorphism for $G$ defined by $\phi(X)=u$.  Then if $Y \in \mathcal{N}(\mathfrak{g})$ is of nilpotent order $p^m$, there is for every $a,b \in k$ some $g \in \mathcal{U}(G)$ of order $<p^m$ such that $\phi(aY+bY)=\phi(aY)\phi(bY)g$.
\end{prop}

\begin{proof}
Write $|u| = p^r$, and let $a,b \in k$.  Since $kX = T.X \cup \{0\}$ and $W$ is stablized by $T$, $\phi$ maps the line $kX$ to the closed subspace $T.u \cup \{1\} = \overline{T.u} \subseteq W$.  We have that $W/W^p \cong \mathbb{G}_a$, and $\overline{T.u}$ clearly maps isomorphically (as a variety) onto this quotient.  We therefore have an isomorphism of varieties from $kX$ to $\mathbb{G}_a$ which sends $0$ to $0$, hence an isomorphism of algebraic groups.  This shows that $\phi(aX)\phi(bX)\phi(-aX-bX) \in W^p$.  As every element in $W$ of order less than $p^r$ is the $p^i$-th power of an element of maximal order for some $i$ (this is noted in the remarks following Theorem \ref{overgroup}), there is some $w \in W$ having the same order as $u$ and $i>0$ such that $$\phi(aX)\phi(bX)\phi(-aX-bX) = w^{p^i}$$
We observe that $w$ is sent to a non-identity element in $W/W^p$, so that there is some $s \in T$ such that $w \in s.uW^p$.  This implies by \cite[Lemma 2.4]{Sei2} that $w$ is in the $G$-orbit of $s.u$ hence in the $G$-orbit of $u$, so by \cite[Lemma 2.2(iii)]{Sei2} we have $C_G(u)=C_G(w)$.  Thus there is a Springer isomorphism $\psi$ with $\psi(X)=w$.

Now let $\widetilde{\phi}$ be the map from $\mathcal{N}(\mathfrak{g})$ to $G$ defined by $$\widetilde{\phi}(Y) = \phi(aY)\phi(bY)\phi(-aY-bY).$$  It is not hard to see that $\widetilde{\phi}$ defines a $G$-equivariant morphism of varieties.  Indeed, viewing $\phi$ as a morphism to $G$ via inclusion, $\widetilde{\phi}$ can be factored as $$\mathcal{N}(\mathfrak{g})  \xrightarrow{f} \mathcal{N}(\mathfrak{g}) \times \mathcal{N}(\mathfrak{g}) \times \mathcal{N}(\mathfrak{g}) \xrightarrow{\phi \times \phi \times \phi} G \times G \times G \xrightarrow{mult.} G$$  with $f(Y) = \left(aY,bY,-aY-bY\right)$, and $G$ acting diagonally on the product varieties.

Let $\psi^{p^i}$ be the morphism from $\mathcal{N}(\mathfrak{g})$ to $G$ given by $\psi^{p^i}(Y) = \psi(Y)^{p^i}$.  In a similar way this is seen to be a $G$-equivariant morphism.  Since $\widetilde{\phi}(X) = w^{p^i} = \psi^{p^i}(X)$ and both maps are $G$-equivariant morphisms, they must be equal on the regular nilpotent orbit, hence by density on all of $\mathcal{N}(\mathfrak{g})$.  Thus, for all $Y \in \mathcal{N}(\mathfrak{g})$, we have that $\phi(aY)\phi(bY)\phi(-aY-bY) = \psi(Y)^{p^i}$.  By Remark \ref{sameorder} if $Y$ has nilpotent order $p^m$ then $\psi(Y)^{p^i}$ is a unipotent element of order $<p^m$.  As the choice of $a,b$ was arbitrary, this proves the proposition.
\end{proof}

These two propositions now prove the following:

\begin{thm}\label{answer}
If $G$ is simple and $p$ is separably good for $G$, then there exists a Springer isomorphism $\phi: \mathcal{N}(\mathfrak{g}) \rightarrow \mathcal{U}(G)$ such that $\phi$ restricts to $\varepsilon_P$ for every restricted parabolic subgroup  $P \le G$.
\end{thm}

\begin{proof}
We may take $\phi$ to be as in Proposition \ref{additive}, possibly composing with a scalar map on $\mathcal{N}$ if needed to ensure that the tangent map is the identity thanks to property (2) in Theorem \ref{differential}.  It follows that if $Y^{[p]}=0$, then $\phi(aY+bY)=\phi(aY)\phi(bY)$, therefore $\phi(kY)$ is a one-parameter additive subgroup of $G$.  As the tangent map of $\phi$ is the identity, we may now apply Proposition \ref{sufficient} which completes the proof.
\end{proof}

\begin{remark}\label{expomono}
Though it is clear from the arguments in this section, we highlight for later use that if $\phi$ restricts to $\varepsilon_P$ for all restricted $P$, then $\phi$ ``exponentiates" the one-parameter subgroups in Theorem \ref{mono}.  That is, if $Y=d\varphi_u(\frac{d}{dt})$, then $\varphi_u(a)=\phi(aY)$.
\end{remark}

\section{An Explicit Isomorphism For Classical Groups}

In this section we show that for classical matrix groups, the existence of a Springer isomorphism restricting to $\varepsilon_P$ can be given explicitly by the Artin-Hasse exponential series.

\bigskip
Let $a = \{a_i\}_{i=1}^{n-1}$ be any sequence of elements in $k$, and consider the map 

$$\phi_a: \mathcal{N}(\mathfrak{gl}_n) \rightarrow \mathcal{U}(GL_n), \quad \phi_a(Y) = 1 + \sum_{i=1}^{n-1} a_i Y^i$$

This map is algebraic, respects the conjugation action of $GL_n$, and thus defines a $GL_n$-equivariant morphism from $\mathcal{N}(\mathfrak{gl}_n)$ to $\mathcal{U}(GL_n)$.  Moreover if $a_1 \ne 0$ and if $X$ is regular nilpotent, then it follows from \cite[6.7(1)]{J} that $a_1X + \sum_{i=2}^{n-1} a_i X^i$ will also be regular nilpotent, so that $\phi_a(X)$ is a regular unipotent element.  This is most easily seen when $X$ is the nilpotent matrix which is a Jordan block of size $n$, and it is then true for any conjugate of $X$.

We see that $C_{GL_n}(X) \subseteq C_{GL_n}(\phi_a(X))$.  By the existence of a Springer isomorphism for $GL_n$, $C_{GL_n}(\phi_a(X))$ is equal to $gC_{GL_n}(X)g^{-1}$ for some $g \in GL_n$.  The inclusion $C_{GL_n}(X) \subseteq gC_{GL_n}(X)g^{-1}$ implies that they are equal as they have the same dimension and are both connected, so we have $C_{GL_n}(X) =  gC_{GL_n}(X)g^{-1} = C_{GL_n}(\phi_a(X))$.  Thus there is a Springer isomorphism $\phi$ which maps $X$ to $\phi_a(X)$, and it must in fact be given by $\phi_a$, since $\phi$ and $\phi_a$ are equal on the regular nilpotent orbit which is open in the irreducible variety $\mathcal{N}(\mathfrak{gl}_n)$.  In this way any sequence $a_1, \ldots, a_{n-1}, a_1 \ne 0$, defines a Springer isomorphism for $GL_n$ (compare with \cite[\S 10]{M2}).

In particular, we may choose a sequence such that $a_i = 1/i!$ for $i<p$.  If $\phi$ is the resulting Springer isomorphism, then for a $[p]$-nilpotent matrix $Y$ we have $$\phi(Y) = 1 + Y + \frac{Y^2}{2} + \cdots + \frac{Y^{p-1}}{(p-1)!}$$ hence $\phi(aY+bY)=\phi(aY)\phi(bY)$ for all $a,b \in k$.  By Proposition \ref{sufficient}, such a sequence will define a Springer isomorphism for $GL_n$ which has our desired restriction property.  If $G$ is one of the classical subgroups of $GL_n$ listed above, however, it is not true in general that $\phi$ will restrict to a Springer isomorphism for $G$.  To ensure this latter property holds, we will work with the sequence given by the Artin-Hasse exponential series.

\begin{prop}\label{restricts}
Let $G$ be either $SO_n$ or $Sp_{n}$, $n=2n^{\prime}$ in the latter case, and identify $G \le GL_n$ via its natural embedding.  Let $\phi$ be the Springer isomorphism for $GL_n$ given by the sending $X \in \mathcal{N}(\mathfrak{gl}_n)$ to $e_p(X)$, where $e_p(t)$ is the image of the Artin-Hasse exponential series in $k[\![ t ]\!]$.  Then $\phi$ restricts to a Springer isomorphism for $G$.
\end{prop}

\begin{proof}
Let $0 \ne X \in \mathcal{N}(\mathfrak{gl}_n)$ be of nilpotent degree $p^m$.  In \cite[Proposition 7.5]{M}, McNinch proves that if $X \in \mathfrak{g}$, then the injective morphism of Theorem \ref{wittmorphism} has image in $G$.  In particular, $e_p(X)$ is the image of $(1,0,\ldots,0)$ under this map, proving the claim.
\end{proof}

\begin{remark}
As noted earlier, the definition of the Artin-Hasse exponential used in \cite{M} is inverse to the one we are using.  Thus, the definition of the map $E_X$ given here would correspond in McNinch's work to the map $E_{-X}$.  As $X \in \mathcal{N}(\mathfrak{g}) \iff -X \in \mathcal{N}(\mathfrak{g})$, the proof holds regardless. 
\end{remark}

\section{Statement of Main Result}

\begin{thm}\label{main}
Let $G$ be a simple algebraic group, and suppose that $p$ is separably good for $G$.  Then there is a Springer isomorphism $\phi: \mathcal{N}(\mathfrak{g}) \xrightarrow{\sim} \mathcal{U}(G)$ such that:

\begin{enumerate}
\item For any restricted parabolic $P \le G$, $\phi$ restricted to $\mathfrak{u}_P$ is $\varepsilon_P$.
\item For all $X \in \mathcal{N}(\mathfrak{g})$, $\phi(X^{[p]}) = \phi(X)^p$.
\item If $X \ne 0$, and $m$ is the least integer such that $X^{[p^{m}]}=0$, then $\phi$ defines an injective morphism $\mathcal{W}_m \rightarrow G$ given by 
\vspace{0.1in}
\begin{center}$(a_0,a_1,\ldots,a_{m-1}) \mapsto \phi(a_0X)\phi(a_1X^{[p]})\cdots \phi(a_{m-1}X^{[p^{m-1}]})$\end{center}
\end{enumerate}
\end{thm}

\begin{proof}
We know by Theorem \ref{answer} that there is some Springer isomorphism satisfying (1) for all such $G$.  However, to show that one exists which satisfies all of the properties above we will split the proof into classical and exceptional cases.

First suppose that $G$ is one of the groups $SL_n, SO_n$, or $Sp_{2n}$, with its natural embedding in $GL_n$ or $GL_{2n}$, and with corresponding Springer isomorphism given by the Artin-Hasse exponential series.  Property (1) then follows from Proposition \ref{sufficient}, while (3) holds for $\phi$ thanks to \cite[Proposition 7.5]{M}.  To see that (2) holds we note that since the coefficients of $e_p(t)$ lie in $\mathbb{F}_p$ we get $e_p(X)^p = e_p(X^p)=e_p(X^{[p]})$  (in this last equality we are using the fact that the embedding of $G$ guarantees that $X^p$ as an element of $\mathfrak{gl}_n$ is equal to the image of $X^{[p]}$).

The assumption that $p$ is separably good ensures that these results will also apply to any group isogenous to one of these classical groups above, so this proves (2) and (3) for classical types.

If $G$ is of exceptional type, let $\phi$ be as in Proposition \ref{additive} (and again, adjusting if necessary by a scalar map on $\mathcal{N}$ so that tangent map is identity).  As $p$ is good for $G$, \cite[0.4]{T} shows that all unipotent elements either have order $p$ or $p^2$, hence all nilpotent elements have nilpotent order $p$ or $p^2$.  However, it is clear that $\phi$ satisfies (2) and (3) if every element is $[p]$-nilpotent, thus we suppose we are in the second case.

Let $X$ and $T$ be as in the construction of $\phi$ in Proposition \ref{additive}.  So $\phi(X)$ is contained in a unique $W \le Z(C_G(\phi(X)))^0$ which by Lemma \ref{isoismorphic} is isomorphic to $\mathcal{W}_2$ and on which $T$ acts without fixed points.  We have $\phi(kX) \subseteq W$, and we may put coordinates $(a_0,a_1)$ on $\mathcal{W}_2$ so that there is an isomorphism $f: \mathcal{W}_2 \rightarrow W$ with $f((a_0,0)) = \phi(a_0X)$.  As in \S \ref{witt}, we have $f((0,a_1)) = \phi(F^{-1}(a_1)X)^p$, where $F^{-1}$ is the inverse of the Frobenius map (and, we note, not an algebraic map).  We also see that if $k[\mathcal{W}_2] = k[t_0,t_1]$, then the differential $df$ maps $\frac{d}{dt_0} \mapsto X$, and therefore sends $\frac{d}{dt_1} \mapsto X^{[p]}$.

Let $u=\phi(X)^p$.  Because $W^p$ is the saturation of $u$, it follows that the unique monomorphism $\varphi_u$ of Theorem \ref{mono} can be given by $\varphi_u(a)=f(0,a)$.  We see then that $$d\varphi_u\left(\frac{d}{dt}\right) = df\left(\frac{d}{dt_1}\right)=X^{[p]}.$$  By Remark \ref{expomono} we have that $\varphi_u(a) = \phi(aX^{[p]})$.  Thus $$\phi(X^{[p]})=\varphi_u(1)=f(0,1)=\phi(X)^p.$$  This shows that (2) and (3) hold for regular elements, and by arguments similar to those in the proof of Proposition \ref{additive}, must hold for all nilpotent elements.   

\end{proof}

We now establish a simple lemma about Springer isomorphisms (possibly shown elsewhere), after which we have as a corollary that \cite[Theorem 3]{CLN} extends to all separably good primes.

\begin{lemma}
If $\phi$ is any Springer isomorphism for $G$, then for any $X,Y \in \mathcal{N}(\mathfrak{g})$, $[X,Y]=0$ if and only if $\phi(X)$ commutes with $\phi(Y)$.
\end{lemma}

\begin{proof}
Under our conditions on $p$, we have by \cite[\S 2.5,2.6]{J} that $C_{\mathfrak{g}}(Y)=\textup{Lie}(C_G(Y))$.  If $[X,Y]=0$, then $X \in C_{\mathfrak{g}}(Y)=\textup{Lie}(C_G(Y))=\textup{Lie}(C_G(\phi(Y))$, therefore $\phi(Y)$ acts trivially via the adjoint action on $X$, and hence $\phi(Y)$ commutes with $\phi(X)$.  Conversely, by the remarks preceding \cite[Theorem E]{MT}, if $\phi(X) \in C_G(\phi(Y))$ then $X \in \textup{Lie}(C_G(\phi(Y))$, hence $X \in \textup{Lie}(C_G(Y)) = C_{\mathfrak{g}}(Y)$. 
\end{proof}

\begin{cor}
Let $G$ and $\phi$ be as in Theorem \ref{main}.  Then $\phi$ restricts to a unique $G$-equivariant isomorphism $\phi_1: \mathcal{N}_1(\mathfrak{g}) \xrightarrow{\sim} \mathcal{U}_1(G)$ of algebraic varieties with the following properties:

\begin{enumerate}
\item For all $0 \ne X \in \mathcal{N}_1(\mathfrak{g})$, $\phi_1$ restricted to $kX$ is a one-parameter subgroup of $G$ which can be extended to a good $A_1$ subgroup.
\item For any $X,Y \in \mathcal{N}_1(\mathfrak{g})$, $[X,Y]=0$ if and only if $\phi_1(X)$ commutes with $\phi_1(Y)$.
\item The isomorphism $\phi_1$ is defined over $\mathbb{F}_p$.
\end{enumerate}
\end{cor}

\begin{remark}\label{CLNexplained}
In \cite{CLN} this result was established for all simple groups $G$ in separably good characteristic provided that $\mathcal{N}_1(\mathfrak{g})$ is a normal variety.  It is not known in general whether this normality condition holds when $p<h$, and the authors observed that this condition could be dropped if, in our notation, $\varepsilon_P$ came from restricting a Springer isomorphism (see the final paragraph of \S 2.7 of \textit{loc. cit.}).  The authors demonstrate the importance of $\phi_1$ (which they call ``exp") in Sections 3 and 4 of \textit{loc. cit.}, using it in a critical way to prove significant results about cohomological support varieties of rational $G$-modules.  This corollary is then a first step in extending their results to smaller primes.
\end{remark}

\bigskip
\noindent \textbf{Acknowledgements:} We wish to acknowledge helpful discussions and comments on various versions of this paper from Zongzhu Lin, George McNinch, Dan Nakano, Arun Ram, and Craig Westerland.  We thank the referees for many good suggestions which have surely improved the exposition in this paper.  This research was partially supported by grants from the Australian Research Council (DP1095831, DP0986774 and DP120101942).

\end{document}